\documentclass[11pt]{amsart}

\usepackage{float}

\usepackage{bbm}
\usepackage{amsfonts}
\usepackage{amsmath}
\usepackage{amssymb}
\usepackage{graphicx}
\usepackage[bindingoffset=0.2in,%
left=1.1in,right=1.1in,top=1.2in,bottom=1.375in,%
footskip=.25in]{geometry}

\usepackage{mathtools}
\usepackage{amsthm}
\usepackage{latexsym}
\usepackage{fancyhdr}
\usepackage{array}
\usepackage{amscd}
\usepackage{lscape}
\usepackage{tikz}
\usepackage{float}
\usepackage{bm}
\usepackage{subfigure}
\usepackage{grffile}
\usepackage{array}
\usepackage{longtable}
\usepackage{booktabs}

\usepackage{ dsfont }
\usepackage{lipsum}
\newcolumntype{C}[1]{>{\centering\arraybackslash}p{#1}}

\definecolor{navy}{HTML}{2F729C} 
\usepackage[hyperfootnotes=false, colorlinks, linkcolor={blue}, citecolor={magenta}, filecolor={blue}, urlcolor={blue}, plainpages=false, pdfpagelabels]{hyperref}
\usepackage[tableposition=below]{caption}
\newcommand{\A}{{\mathbb A}}
\newcommand{\Q}{{\mathbb Q}}
\newcommand{\Z}{{\mathbb Z}}

\newcommand{\C}{{\mathbb C}}

\newcommand{\p}{\mathfrak p}
\newcommand{\OF}{{\mathfrak o}}
\newcommand{\GL}{{\rm GL}}

\newcommand{\Gal}{{\rm Gal}}

\newtheorem{theorem}{Theorem}[section]
\newtheorem{lemma}[theorem]{Lemma}

\theoremstyle{definition}

\newenvironment{claim}[1][Claim]{\begin{trivlist}
\item[\hskip \labelsep {\bfseries #1}]}{\end{trivlist}}

\theoremstyle{remark}

\numberwithin{equation}{section}

\makeatletter
\@namedef{subjclassname@2020}{\textup{2020} Mathematics Subject Classification}
\makeatother
\begin{document}

\title[Representations attached to elliptic curves]{Representations attached to elliptic curves \\with a non-trivial odd torsion point}

\author{Alexander J. Barrios}
\address{Department of Mathematics and Statistics, Carleton College, Northfield, Minnesota}
\email{abarrios@carleton.edu}

\author{Manami Roy}
\address{Department of Mathematics, Fordham University, Bronx, New York 10458}
\email{mroy17@fordham.edu}

\subjclass[2020]{Primary 11G05, 11G07, 11F70}

\keywords{Elliptic Curves, multiplicative reduction, automorphic representations.}

\begin{abstract}
We give a classification of the cuspidal automorphic representations attached to rational elliptic curves with a non-trivial torsion point of odd order. Such elliptic curves are parameterizable, and in this paper, we find the necessary and sufficient conditions on the parameters to determine when split or non-split multiplicative reduction occurs. Using this and the known results on when additive reduction occurs for these parametrized curves, we classify the automorphic representations in terms of the parameters.
\end{abstract}
\maketitle

\section{Introduction}
Given an elliptic curve $E/\mathbb{Q}$, there is a cuspidal
automorphic representation $\pi\cong\otimes_{p\le \infty}\pi_{p}$ of
$\operatorname*{GL}\!\left(  2,\mathbb{A}_{\mathbb{Q}}\right)$ attached to $E$, where $\pi_{p}$ is the representation of $\GL(2,\Q_p)$ associated to $E/\Q_p$ \cite{Gelbart1976}. In this article, we consider the case when $E$ has a non-trivial torsion point of odd order and explicitly classify the cuspidal automorphic
representation $\pi$ attached to $E$.
Our classification is attained by studying parameterized families of
elliptic curves. Specifically, let $T$ be one of the eight torsion subgroups with a non-trivial point of odd order allowed by Mazur's Torsion Theorem \cite{Mazur1977}. Then there is a family of elliptic curves $E_{T}/\Q$ (see Table \ref{ta:ETmodel}) with the property that if $E$ is an elliptic curve such that $T\hookrightarrow E(\mathbb{Q})$, then $E$ is $\mathbb{Q}$-isomorphic to a member of the parameterized family $E_{T}$. We note that for $T=C_{3}$, we must consider two different families of elliptic curves.

For the parameterized family $E_{T}$, there are necessary and
sufficient conditions on the parameters of $E_{T}$ which determine the minimal discriminant and primes at which additive reduction occurs \cite{Barrios2020}. 
In addition, the conductor exponent and N\'{e}ron type at each prime for which $E_T$ has additive reduction can be determined from the parameters of $E_T$ \cite{BarRoy}. To get a complete classification of the automorphic
representation attached to $E_T$, we require knowledge of when $E_T$ has split or non-split multiplicative reduction. We consider this case in Section~\ref{Sec multi}, where we prove the following result (see Theorem \ref{ThmMultRed}):
\begin{claim}
[Theorem 1.]\textit{There are necessary and sufficient conditions of the parameters of
$E_{T}/\Q$ to determine the primes at which split and non-split multiplicative
reduction occur.}
\end{claim}
In particular, \cite[Main Theorem]{BarRoy} and Theorem \ref{ThmMultRed} give us the conditions on the parameters of $E_T$ which determine the N\'{e}ron type and conductor exponent at each prime. 
Using this information about $E_T$, we find the cuspidal automorphic representation attached to $E_T$ in Section~\ref{sectionreps}.
We do this by using a description of the local representation of
$\operatorname*{GL}\!\left(  2,
\mathbb{Q}
_{p}\right) $ attached to an elliptic curve over $\mathbb{Q}_p$ given in terms of its Weierstrass coefficients (see \cite[Section~2]{roy2019paramodular}). With this, we can state our main theorem.

\begin{claim}
[Theorem 2.]\textit{There are necessary and sufficient conditions of the parameters of
$E_{T}/\Q$ to determine the cuspidal automorphic representation $\pi\cong\otimes_{p\le \infty}%
\pi_{p}$ of $\operatorname*{GL}\!\left(  2,\mathbb{A}_{\mathbb{Q}
}\right)$ attached to~$E_T$.}
\end{claim}
This is a consequence of Theorems~\ref{ThmforC3}-\ref{ThmotherTs} since we consider some of the parameterized families~$E_{T}$ separately. Next, we give a brief review of the necessary background to prove our results.

\subsection{Elliptic curves and parameterizations.} We start by giving a brief review of elliptic curves. For further details, see \cite{Silverman2009} and \cite{Silverman1994}.
Let $E$ be an elliptic curve over a field $K$ given by the Weierstrass model
{
	\begin{equation}
		E:y^{2}+a_{1}xy+a_{3}y=x^{3}+a_{2}x^{2}+a_{4}x+a_{6}.\label{ch:inintroweier}%
	\end{equation}
	with each Weierstrass coefficient $a_{j} \in K$. We define the following quantities
	\begin{equation}
		\begin{split}
			\label{basicformulas}
		c_{4}&=a_{1}^{4}+8a_{1}^{2}a_{2}-24a_{1}a_{3}+16a_{2}^{2}-48a_{4},\\
c_{6}&=-\left(  a_{1}^{2}+4a_{2}\right)  ^{3}+36\left(  a_{1}^{2}%
+4a_{2}\right)  \left(  2a_{4}+a_{1}a_{3}\right)  -216\left(  a_{3}^{2}%
+4a_{6}\right), \\ 
			\Delta&=\frac{c_4^3-c_6^2}{1728},\qquad\quad  j=\frac{c_{4}^{3}}{\Delta}.
		\end{split}
\end{equation}}
We call $c_{4}$ and $c_{6}$ \textit{the invariants associated to the Weierstrass model} of $E$, $\Delta$ is the \emph{discriminant}, and $j$ is the \emph{$j$-invariant} of $E$. 

For each prime $p$, let $v_p$ denote the $p$-adic valuation of $\mathbb{Q}_p$. Each elliptic curve $E$ over~$\Q_{p}$ is $\Q_{p}$-isomorphic to an elliptic curve given by a model of the form \eqref{ch:inintroweier} such that each~$a_{j}\in\mathbb{Z}_{p}$ and 
$v_{p}\!\left(  \Delta\right)\ge0 $ is minimal.
We call this curve a \textit{minimal model} for $E$ and the associated discriminant of this elliptic curve is the \textit{minimal discriminant} of $E$. Next, let $c_4$ and $c_6$ be the invariants associated to a minimal model of $E/\Q_{p}$ and let $\Delta$ denote its minimal discriminant. Then $E$ is said to have

\begin{enumerate}
    \item[(i)] \textit{good reduction}  if $v_p(\Delta)=0.$
\item[(ii)] \textit{multiplicative reduction} if $v_p(\Delta)>0$  and  $v_p(c_4)=0$.
\item[(iii)]  \textit{additive reduction}  if $v_p(\Delta)>0$ and $v_p(c_4)>0$.
\item[(iv)]  \textit{potentially multiplicative reduction} if $v_p(j)<0$.
\item[(v)]  \textit{potentially good reduction} if $v_p(j)\geq0$.
\end{enumerate}
Note that when $E/\mathbb{Q}_p$ has multiplicative reduction, we say that the reduction is \textit{split} if the slopes of the tangent lines at the node are in $\mathbb{F}_p$. Otherwise, it is said to be \textit{nonsplit}. 
When $E/\Q_p$ has potentially multiplicative reduction, we define the \emph{$\gamma$-invariant} of $E/\Q_p$~by
\begin{equation}
\gamma(E/\Q_p)=-\frac{c_4}{c_6} \in \Q_p^{\times}/\Q_p^{\times2}.
\end{equation}
This quantity is well defined and independent of the choice of the Weierstrass equation. 

For an elliptic curve $E/\mathbb{Q}$, we define the
\textit{conductor} $N_{E}$ to be the quantity $N_{E}=\prod_{p}p^{f_{p}}$ where $f_{p}$ is a nonnegative integer called \textit{the conductor exponent at} $p$ and it is computed via Tate's Algorithm \cite{Tate1975}. We note that when $f_p=0$, the elliptic curve has good reduction over $\Q_p$.
Now suppose further that $E$ has a non-trivial torsion point of odd order. By Mazur's Torsion Theorem, $E(\mathbb{Q})_{\text{tors}}$ is isomorphic to one of eight possible torsion subgroups~$T$. Moreover, $E$ occurs in a parameterized family of elliptic curves \cite[Table~3]{Kubert1976}. In this article, we consider a modification of these parameterized families, namely the family elliptic curves $E_T=E_T(a,b)$ where the Weierstrass model of $E_{T}$ is as given in Table~\ref{ta:ETmodel}. By \cite[Proposition~4.3]{Barrios2020}, there are relatively prime integers $a$ and $b$ such that $E$ is $\mathbb{Q}$-isomorphic to~$E_T$.

{\renewcommand*{\arraystretch}{1.3}
\begin{longtable}{C{0.6in}C{1.2in}C{1.7in}C{1.9in}}
\hline
$T$ & $a_{1}$ & $a_{2}$ & $a_{3}$  \\
\hline
\endfirsthead
\caption[]{\emph{continued}}\\
\hline
$T$ & $a_{1}$ & $a_{2}$ & $a_{3}$ \\
\hline
\endhead
\hline
\multicolumn{2}{r}{\emph{continued on next page}}
\endfoot
\hline
\caption[Weierstrass Model of $E_{T}$]{The Weierstrass Model $E_{T}:y^{2}+a_{1}xy+a_{3}y=x^{3}+a_{2}x^{2}$}\label{ta:ETmodel}
\endlastfoot
$C_{3}^{0}$ & $0$ & $0$ & $a$  \\\hline
$C_{3}$ & $a$ & $0$ & $a^{2}b$ \\\hline
$C_{5}$ & $a-b$ & $-ab$ & $-a^{2}b$  \\\hline
$C_{6}$ & $a-b$ & $-ab-b^{2}$ & $-a^{2}b-ab^{2}$ \\\hline
$C_{7}$ & $a^{2}+ab-b^{2}$ & $a^{2}b^{2}-ab^{3}$ & $a^{4}b^{2}-a^{3}b^{3}$  \\\hline
$C_{9}$ & $a^{3}+ab^{2}-b^{3}$ & $
a^{4}b^{2}-2a^{3}b^{3}+
2a^{2}b^{4}-ab^{5}
$ & $a^{3}\cdot a_{2}$
 \\\hline
$C_{10}$ &$
a^{3}-2a^{2}b-
2ab^{2}+2b^{3}
$ & $-a^{3}b^{3}+3a^{2}b^{4}-2ab^{5}$ & $(a^{3}-3a^{2}b+ab^{2})\cdot a_{2}$ \\\hline
$C_{12}$ & $
-a^{4}+2a^{3}b+2a^{2}b^{2}-
8ab^{3}+6b^{4}
$ & $b(a-2b)(a-b)^{2}(a^{2}-3ab+3b^{2})(a^{2}-2ab+2b^{2})
$ & $a(b-a)^3 \cdot a_{2} $ \\\hline
$C_{2}\times C_{6}$ & $-19a^{2}+2ab+b^{2}$ & $
-10a^{4}+22a^{3}b-
14a^{2}b^{2}+2ab^{3}
$ & $
90a^{6}-198a^{5}b+116a^{4}b^{2}+
4a^{3}b^{3}-14a^{2}b^{4}+2ab^{5}
$  
\end{longtable}}
\vspace{-1em}
\subsection{Representations of \texorpdfstring{$\GL(2,\Q_p)$}{}}
\label{local rep}
Let $G=\GL(2,\Q_p)$. In this section we briefly review various types of irreducible, admissible representations of $G$. For a general reference of this section, see \cite{JacquetLanglands1970} and \cite[Section~4]{Bump1997}. Let $\chi_1,\chi_2$ be two characters of $\Q_p^{\times}$. Let $V(\chi_1,\chi_2)$ be the space of the standard induced representations $\pi=\pi(\chi_1,\chi_2)$ of $G$ consisting of all locally constant functions $f:G \rightarrow \C$ with the
property $$f\left( \begin{bmatrix}a&b\\0&d\end{bmatrix}g\right)=|ad^{-1}|^{1/2}\chi_1(a)\chi_2(d)f(g)\quad \text{ for all } g \in G,\ a, d \in \Q_p^{\times}, b\in \Q_p.$$
The action of $G$ on $V(\chi_1,\chi_2)$ is by right translation,  i.e., $$g \cdot f(x) = f(xg), \text{ for } g, x \in G \text{ and } f \in V(\chi_1,\chi_2).$$

It is a well-known that $\pi(\chi_1,\chi_2)$ is irreducible if and only if $\chi_1\chi_2^{-1}\neq |\cdot|^{\pm 1}$. In this case
$\pi$ is called a \emph{principal series representation} which we denote by $\chi_1 \times \chi_2$.

The representation $\pi( |\cdot|^{1/2},  |\cdot|^{-1/2})$ is not irreducible and has two constituents. The unique
irreducible quotient is the trivial representation. The unique irreducible subrepresentation is called the
\emph{Steinberg representation} which we denote by $\text{St}_{\GL(2,\Q_p)}$.  Moreover, for a character~$\chi$ of $\Q_p^{\times}$, the representation $\pi(\chi |\cdot|^{1/2},  \chi|\cdot|^{-1/2})$ has a unique irreducible subrepresentation which is just the \emph{twist of the Steinberg representation} denoted by $\chi \text{St}_{\GL(2,\Q_p)}$. The quotient in this case is the one-dimensional representation $\chi \circ \text{ det}$. The representations $\chi \text{St}_{\GL(2,\Q_p)}$ are also known as \emph{special representations}.

Every irreducible representation that is not a subquotient of some $\pi(\chi_1,\chi_2)$ is called \emph{supercuspidal}. There is a special type of supercuspidal representation associated to a quadratic field extension  $F/\Q_p$ and a character $\xi$ of $F^{\times}$ such that it is not trivial on the kernel of the norm map $N_{F/\Q_p}:F^{\times}\rightarrow \Q_p^{\times}$. Equivalently, $\xi\neq \xi^{\sigma}$ (where $\xi^{\sigma}(x)=\xi(\sigma(x))$) for $\sigma \in \Gal(F/\Q_p)$ with $\sigma \neq 
\mathbbm{1}$. We call this a \emph{dihedral supercuspidal representation} and we denote this by~$\omega_{F, \xi}$. For a reference of the construction, see  \cite[Section~1.1]{JacquetLanglands1970} or \cite[Section~4.8]{Bump1997}. 

For an irreducible admissible representation $\pi$ of $\GL(2,\Q_p)$, we have $\pi\cong \omega_{F, \xi}$ for some $\xi$ if and only if $\pi\cong \pi\otimes \chi_{F/\Q_p}$ where $\chi_{F/\Q_p}$ is the quadratic character attached to $F/\Q_p$ (see \cite{GerardinLi1989}).
For an odd prime $p$, every supercuspidal representation of $\GL(2, \Q_p)$ is isomorphic to some $\omega_{F, \xi}$. 

For $n \ge 0$, let
$$\Gamma_2(n)=\left\{ \begin{bmatrix}a&b\\c&d\end{bmatrix}\in \GL(2,\Z_p): c\in p^{n}\Z_p,\ d \in 1+p^{n}\Z_p \right\}.$$

\noindent Let $(\pi,V)$ be an infinite-dimensional irreducible admissible representation of
$G$.  Let $V(m)$ denote the space of $\Gamma_2(m)$-fixed vectors in $V$. Let $n \ge 0$ be the smallest non-negative integer such that $V(n)\neq \{0\}$. Then we say the \emph{conductor} of $\pi$ is $a(\pi) = n$.

For a character $\chi$ of $\Q_p^{\times}$, the smallest $n\ge 0$ such that $\chi|_{1+p^n\Z_p}=~1$ is called the \emph{conductor} of $\chi$ and is denoted by $a(\chi)$. If $a(\chi)=0$ we say $\chi$ is \emph{unramified}; otherwise $\chi$ is \emph{ramified}. 
Also, recall that the \emph{central character} $\omega_{\pi}$ of $\pi$ is the character of the center $Z\cong \Q_p^{\times}$ of $G$ satisfying 
\begin{center}
$\pi(zg)=\omega_{\pi}(z)\pi(g)$ for $z\in Z$ and $g\in G$.
\end{center}
For the purpose of this paper we consider the irreducible admissible representations $\pi$ of $G$ with trivial central character. The following table lists the conductors $a(\pi)$ for such representations.
\begin{equation}
\label{conductor}
	\renewcommand{\arraystretch}{0.6}
	\renewcommand{\arraycolsep}{.4cm}
	 \begin{array}{ccc} 
		\toprule
		\pi&\text{Condition\ on } \pi&a(\pi)\\
		\toprule
		\chi \times \chi^{-1}&&a(\chi) + a(\chi^{-1})\\  
		\midrule
		\chi \text{St}_{\GL(2,\Q_p)}& \chi \text{ is ramified}&2a(\chi) \\
		\cmidrule{2-3}
		\chi^2=1&\chi \text{ is unramified} &1\\  
		\midrule
			\omega_{F,\xi}&&f(F/\Q_p)a(\xi)+d(F/\Q_p)\\
		\bottomrule\\
	\end{array}
\end{equation}

\noindent Here $f(F/\Q_p)$ is the residue class degree, $d(F/\Q_p)$ is the valuation of the discriminant of the field extension $F/\Q_p$. Let $\OF_F$ be the ring of integers of $F$ and let $\p$ be its maximal ideal. Then $a(\xi)$ is the conductor of $\xi$, i.e., $a(\xi)$ is the smallest $n \ge 0$ such that $\xi|_{1+\p^n\OF_F}=1$. Hence, for $p$ odd, we have
\begin{equation}
	\label{conductor_of_sc}
	a(\omega_{F,\xi})=\begin{cases}
		2a(\xi)\qquad &\text{ if }  F/\Q_p \text{ is unramified},\\
		1+a(\xi)\qquad &\text{ if }  F/\Q_p \text{ is ramified}.
	\end{cases}
\end{equation}
	The central character of $\omega_{F, \xi}$ is 
	$\xi|_{\Q_p^{\times}}\cdot \chi_{F/\Q_p}$, where $\chi_{F/\Q_p}$ is the quadratic character of~$\Q_p^{\times}$ associated to the quadratic extension $F/\Q_p$ such that $\chi_{F/\Q_p}\left({N_{F/\Q_p}(F^{\times})}\right)=1$.  If $\omega_{F, \xi}$ has trivial central character, i.e., $\xi|_{\Q_p^{\times}}\cdot \chi_{F/\Q_p}=1$, then by evaluating $\xi$ at $N_{F/\Q_p}(y)$ for any~$y \in F^{\times}$, we get $\xi^{\sigma}=\xi^{-1}$ on~$F^{\times}$. Next, we consider the following useful result about representations of conductor $2$ of $\GL(2,\Q_2)$.

\begin{lemma}
	\label{rep of conductor 2}
Let $F=\Q_2(\sqrt{5})$ be the unramified quadratic extension of $\Q_2$. There is a unique representation of $\GL(2,\Q_2)$ with trivial central character and conductor $2$ up to isomorphism. This is the dihedral supercuspidal  representation $\omega_{F, \xi}$ with $a(\xi)=1$.
\end{lemma}
\begin{proof}
Let $\pi$ be a representation of $\GL(2,\Q_2)$ such that $a(\pi)=2$. Since $\Z_2^{\times}=1+2\Z_2$, there is no character of conductor $1$ of $\Q_2^{\times}$. Then, it follows from \eqref{conductor} that $\pi$ has to be a supercuspidal representation. Since $a(\pi)=2$ is minimal, by \cite[Proposition 3.5]{Tunnell1978}, $\pi\cong \omega_{F,\xi}$ where $F=\Q_2(\sqrt{5})$ is the unramified quadratic extension of $\Q_2$. In this case we have $a(\pi)=2a(\xi)$, so $a(\xi)=1$. Since $a(\xi)=1$, we get the induced character $\xi: \OF_{F}^{\times}/(1+\p\OF_F) \rightarrow \C^\times$. Note that $\OF_{F}^{\times}/(1+\p\OF_F)$ is a cyclic group of order $3$. Now, there are exactly $2$ elements of order~$3$ in $\OF_{F}^{\times}/(1+\p\OF_F)$ which are inverses of each other. Since $\omega_{F, \xi}$ has trivial central character by assumption, we have $\xi^{\sigma}=\xi^{-1}$. So, there are exactly two such characters in this case, which are Galois conjugates of each other. Hence the result follows since $\omega_{F, \xi} \cong \omega_{F, \xi^{\sigma}}$.
\end{proof}

Given an elliptic curve $E/\Q_p$, there is a local representation $\pi_p$ of $\GL(2,\Q_p)$. In fact, $\pi_p$ is one of the local representations given in \eqref{conductor}. In Section~\ref{sectionreps}, we explicitly find the local representations $\pi_p$ attached to $E_T/\Q_p$, and consequently, we attain the cuspidal automorphic representation attached to $E_T/\Q$.

\section{Multiplicative Reduction}\label{Sec multi}
Let $E_T$ be as given in Table~\ref{ta:ETmodel}. The minimal discriminant of
$E_{T}$ is $\Delta_{T}=u_{T}^{-12}\gamma_{T}$ where~$u_T$ and $\gamma_{T}$ are as given
in Table~\ref{mindiscs} \cite[Theorem~4.4]{Barrios2020}. When $T=C_{3}$, we write $a=c^{3}d^{2}e$ for~$c,d,e$
integers such that $d,e$ are relatively prime positive squarefree integers. 

{\begingroup

\renewcommand{\arraystretch}{1.2}
 \begin{longtable}{ccc}
	\hline
	$T$ & $\gamma_{T}$ & $u_{T}$\\
	\hline

	\endfirsthead
	\hline
	$T$ & $\gamma_{T}$ & $u_{T}$\\
	\hline
	\endhead
	\hline
	
	\multicolumn{3}{r}{\emph{continued on next page}}
	\endfoot
	\hline 
	\caption{The minimal discriminant $\Delta_T=u_T^{-12}\gamma_T$ of $E_T$}
	\endlastfoot

$C_{3}$ & $d^{4}e^{8}b^{3}(  a-27b)$ &$c^2d$ \\\hline
$C_{3}^0$ & $-27a^4$ & $1$ \\\hline
$C_{5}$ & $-a^{5}b^{5}(  a^{2}+11ab-b^{2})$&$1$  \\\hline
$C_{6}$ & $a^{2}b^{6}(  a+9b)  (  a+b)  ^{3}$ &$1\text{ if }v_2(a+b)\leq2$ \\
& & $2\text{ if }v_2(a+b)\geq3$ \\\hline
$C_{7}$ & $-a^{7}b^{7}(  a-b)  ^{7}(  a^{3}+5a^{2}%
b-8ab^{2}+b^{3})$ &$1$ \\\hline
$C_{9}$ & $-(  a^2b-ab^2)  ^{9}(  a^{2}-ab+b^{2})
^{3} (  a^{3}+3a^{2}b-6ab^{2}+b^{3})$ &$1$ \\\hline
$C_{10}$ & $(a^{5}b^{10}(  a-b)  ^{10}(  a-2b)
^{5}$  &$1\text{ if }v_2(a)=0$ \\
& $(  a^{2}+2ab-4b^{2})(  a^{2}-3ab+b^{2})  ^{2})$ & $2\text{ if }v_2(a)\geq1$ \\\hline
$C_{12}$ & $(a^{2}(  ab-b^2)  ^{12}(  a-2b)
^{6}(  a^{2}-2ab+2b^{2})  ^{2}$ &$1\text{ if }v_2(a)=0$\\
& $(  a^{2}-2ab+2b^{2})
^{3}(  a^{2}-3ab+3b^{2})  ^{4})$ & $2\text{ if }v_2(a)\geq1$ \\\hline
$C_{2}\times C_{6}$ & $((  2a)  ^{6}(  b-a)  ^{6}(
b-9a)  ^{2}$ &$1\text{ if }v_2(a+b)=0$\\
&$(  b^{2}-9a^{2})  ^{2}(  b-5a)  ^{6})$&$4\text{ if }v_2(a+b)\geq2 $\\
&&$16\text{ if }v_2(a+b)=1$

\label{mindiscs}
\end{longtable}
\endgroup}\vspace{-0.7em}

Our first result establishes the necessary and sufficient conditions on the parameters of $E_T$ to determine when $E_T$ has multiplicative reduction. We observe that $E_{C_{3}^0}$ has $j$-invariant $0$ and thus $E_{C_{3}^0}$ has additive reduction at each prime dividing the minimal discriminant. In particular, there are no primes at which multiplicative reduction occurs for $E_{C_{3}^0}$.

\begin{lemma}
\label{LemmaMult} Let $E_T$ be as given in Table \ref{ta:ETmodel} for $T\neq C_3^0$. Then
$E_{T}$ has multiplicative reduction at a prime~$p$ if and only if the
parameters of $E_{T}$ satisfy one of the conditions listed in
Table~\ref{ta:multred0} at~$p$.
\end{lemma}
\vspace{-1em}
{\begingroup

\renewcommand{\arraystretch}{1.2}
 \begin{longtable}{ccc}
	\hline
	$T$ & $p$ & Conditions on the parameters of $E_{T}$\\
	\hline

	\endfirsthead
	\hline
	$T$ & $p$ & Conditions on the parameters of $E_{T}$\\
	\hline
	\endhead
	\hline

	\multicolumn{3}{r}{\emph{continued on next page}}
	\endfoot
	\hline 
	\caption{Primes at which $E_T$ has multiplicative reduction}
	\endlastfoot

$C_{3}$ & $\geq2$ & $v_{p}\!\left(  b\right)  >0$\\\cline{2-3}
& $\neq3$ & $v_{p}\!\left(  a-27b\right)  >0$\\\hline
$C_{5}$ & $\geq2$ & $v_{p}\!\left(  ab\right)  >0$\\\cline{2-3}
& $\geq7$ & $v_{p}\!\left(  a^{2}+11ab-b^{2}\right)  >0$\\\hline
$C_{6}$ & $2$ & $v_{2}\!\left(  a+b\right)  \geq3$\\\cline{2-3}
& $3$ & $v_{3}\!\left(  a+b\right)  >0$ with $v_{3}\!\left(  a\right)  =0$\\\cline{2-3}
& $\geq2$ & $v_{p}\!\left(  b\right)  >0$\\\cline{2-3}
& $\neq3$ & $v_{p}\!\left(  a\right)  >0$\\\cline{2-3}
& $\geq5$ & $v_{p}\!\left(  \left(  a+b\right)  \left(  a+9b\right)  \right)
>0$\\\hline
$C_{7}$ & $\geq2$ & $v_{p}\!\left(  ab\left(  a-b\right)  \right)  >0$\\\cline{2-3}
& $\geq 13$ & $v_{p}\!\left(  a^{3}+5a^{2}b-8ab^{2}+b^{3}\right)  >0$\\\hline
$C_{9}$ & $\geq2$ & $v_{p}\!\left(  ab\left(  a-b\right)  \right)  >0$\\\cline{2-3}
& $\geq7$ & $v_{p}\!\left(  \left(  a^{2}-ab+b^{2}\right)  \left(
a^{3}+3a^{2}b-6ab^{2}+b^{3}\right)  \right)  >0$\\\hline
$C_{10}$ & $\geq2$ & $v_{p}\!\left(  ab\left(  a-b\right)  \left(
a-2b\right)  \right)  >0$\\\cline{2-3}
& $\geq7$ & $v_{p}\!\left(  \left(  a^{2}+2ab-4b^{2}\right)  \left(
a^{2}-3ab+b^{2}\right)  \right)  >0$\\\hline
$C_{12}$ & $\geq2$ & $v_{p}\!\left(  ab\left(  a-b\right)  \left(
a-2b\right)  \right)  >0$ with $v_{3}\!\left(  a\right)  =0$ if $p=3$\\\cline{2-3}
& $\geq5$ & $v_{p}((a^{2}-2ab+2b^{2})(a^{2}-2ab+2b^{2})(a^{2}-3ab+3b^{2}))$\\\hline
$C_{2}\times C_{6}$ & $\geq2$ & $v_{p}\!\left(  2a\left(  b-a\right)  \left(
b-5a\right)  \right)  >0$ with $v_{3}\!\left(  b\right)  =0$ if $p=3$\\\cline{2-3}
& $\geq5$ & $v_{p}\!\left(  \left(  b-9a\right)  \left(  b^{2}-9a^{2}\right)
\right)  >0$
\label{ta:multred0}
\end{longtable}
\endgroup}
\vspace{-1em}
\begin{proof}
Observe that $E_{T}$ has multiplicative reduction at a prime $p$ if and
only if $v_{p}\!\left(  \Delta_{T}\right)  >0$ and $E_{T}$ does not have
additive reduction at $p$. By \cite[Theorem~7.1]{Barrios2020}, there are
necessary and sufficient conditions on the parameters of $E_{T}$ to determine the primes at which additive reduction occur. Below, we assume this result implicitly. Observe that the result follows for~$T=C_3,C_6,C_{12},C_2\times C_6$ from loc cit. We now consider the remaining cases, and note that $a$ and $b$ are assumed to be relatively prime integers.

\textbf{Case 1. }Suppose $T=C_{5}$. Then $E_{T}$ has additive reduction at a
prime $p$ if and only if~$p=5$ and $v_{5}\!\left(  a+3b\right)  >0$. This is equivalent to $v_{5}\!\left(  a^{2}%
+11ab-b^{2}\right)  >0$. The claim now follows from Table~\ref{mindiscs} since $v_p(a^{2}+11ab-b^{2})=0$ for $p=2,3$.

\textbf{Case 2.} Suppose $T=C_{7}$. Then $E_{T}$ has additive reduction at a
prime $p$ if and only if~$p=7$ and $v_{7}\!\left(  a+4b\right)  >0$. This is equivalent to $v_{7}\!\left(  a^{3}%
+5a^{2}b-8ab^{2}+b^{3}\right)  >0$. The lemma now follows from Table~\ref{mindiscs} since $v_p(a^{3}%
+5a^{2}b-8ab^{2}+b^{3})=0$ for $p=2,3,5,11$.

\textbf{Case 3.} Suppose $T=C_{9}$. Then $E_{T}$ has additive reduction at a
prime $p$ if and only if~$p=3$ and $v_{3}(  a+b)  >0$. This is equivalent to $v_{3}((  a^{2}-ab+b^{2}) (
a^{3}+3a^{2}b-6ab^{2}+b^{3}))  >0$. Moreover, for
$p=2,5$, we have
$v_{p}\!\left(  \left(  a^{2}-ab+b^{2}\right)  \left(  a^{3}+3a^{2}%
b-6ab^{2}+b^{3}\right)  \right)  =0$. The lemma now follows for this case by Table~\ref{mindiscs}.

\textbf{Case 4.} Suppose $T=C_{10}$. Then $E_{T}$ has additive reduction at a
prime $p$ if and only if~$p=5$ and $v_{5}\!\left(  a+b\right)  >0$. This is equivalent to $v_{5}( (  a^{2}+2ab-4b^{2}) (  a^{2}-3ab+b^{2}))  >0$. In fact, $v_{p}\!\left(  \left(  a^{2}+2ab-4b^{2}
\right)  \left(  a^{2}-3ab+b^{2}\right)  \right)  =0$ for $p=3$. The same holds for $p=2$, provided that $v_2(a)=0$. The lemma now follows for this case by Table~\ref{mindiscs}.
\end{proof}

Next, we consider the type of multiplicative reduction of $E_T$. To this end, let $F_{T,j}$ be the elliptic curve attained from $E_{T}$ via the admissible change of
variables $x\longmapsto u_{j}^{2}x+r_{j}$ and $y\longmapsto u_{j}^{3}%
y+u_{j}^{2}s_{j}x+w_{j}$, where $u_{j},r_{j},s_{j},$ and $w_{j}$ are given in terms of $z_{T,j}$, and these quantities are listed in
Table \ref{modelsmultred}. 
When $T$ is fixed, we write $z_j=z_{T,j}$ in the table.
In the theorem below, we will apply Tate's Algorithm to the models $F_{T,j}$ to deduce the type of multiplicative
reduction. We note that the results in the following theorem have been verified on SageMath \cite{sagemath} and can be found at \cite{GitHubLocalRep}.


{\begingroup

\renewcommand{\arraystretch}{1.3}
 \begin{longtable}{ccccccc}
	\hline
	$T$ & $j$ & $z_{T,j}$ & $u_{j}$ & $r_{j}$ & $s_{j}$ & $w_{j}$\\
	\hline
	\endfirsthead
	\hline
	$T$ & $j$ & $z_{T,j}$ & $u_{j}$ & $r_{j}$ & $s_{j}$ & $w_{j}$\\
	\hline
	\endhead
	\hline

	\multicolumn{4}{r}{\emph{continued on next page}}
	\endfoot
	\hline
	\caption{Change of variables $(x,y) \mapsto (u_j^2x+r_j,u_j^3y+u_j^2s_j+w_j)$ to attain $F_{T,j}$ from $E_T$}\label{modelsmultred}
	\endlastfoot

		$C_{3}$ & $1$ & $c^{2}d$ & $z_{1}$ & $z_{1}^{2}b$ & $0$ & $0$\\\cline{2-7}
		& $2$ & $\frac{c^{2}d}{3}$ & $z_{2}$ & $-c^{2}d^{2}e^{2}z_{2}^{2}$ & $-c^{3}d^{2}e$ &
		$c^{3}d^{3}e^{3}z_{2}^{3}$\\\hline
		$C_{5}$ & $1$ & $ab$ & $1$ & $z_{1}$ & $z_{1}$ & $z_{1}$\\\hline
		$C_{6}$ & $1$ & $ab$ & $1$ & $z_{1}$ & $z_{1}$ & $z_{1}$\\\cline{2-7}
		& $2$ & $3b-a$ & $\frac{1}{3}$ & $\frac{az_{2}}{9}$ & $2b$ & $\frac{2a}{27}(a+3b)^2$\\\cline{2-7}
		& $3$ & $a+b$ & $1$ & $z_{3}$ & $0$ & $0$\\\cline{2-7}
		& $4$ & $1$ & $2$ & $4$ & $0$ & $8$\\\cline{2-7}
		 & $ 5 $ & $ a+b $ & $ 2 $ & $ 4z_{5} $ & $ 0 $ & $ 0$\\\hline
		$C_{7} $ & $ 1 $ & $ ab\left(  a-b\right)   $ & $ 1 $ & $ z_{1} $ & $ z_{1} $ & $ 0$\\\hline
		$C_{9} $ & $ 1 $ & $ ab\left(  a-b\right)   $ & $ 1 $ & $ z_{1} $ & $ abz_{1} $ & $ 0$\\\hline
		$C_{10} $ & $ 1 $ & $ ab\left(  a-b\right)  \left(  a-2b\right)   $ & $ 1 $ & $ z_{1} $ & $ z_{1} $ & $
		0$\\\cline{2-7}
		& $2 $ & $ ab\left(  a-b\right)  \left(  a-2b\right)   $ & $ 2 $ & $ 4z_{2} $ & $ 8z_{2} $ & $
		0$\\\hline
		$C_{12} $ & $ 1 $ & $ ab\left(  a-b\right)  \left(  a-2b\right)   $ & $ 1 $ & $ z_{1} $ & $ z_{1} $ & $
		0$\\\cline{2-7}
		& $2 $ & $ ab\left(  a-b\right)  \left(  a-2b\right)   $ & $ 2 $ & $ 4z_{2} $ & $ 8z_{2} $ & $
		0$\\\hline
		$C_{2}\times C_{6} $ & $ 1 $ & $ \left(  b^{2}-9a^{2}\right)  \left(  b-5a\right)   $ & $
		1 $ & $ 2a\left(  b-a\right)z_{1}   $ & $ 2a\left(  b-a\right)z_{1}   $ & $ 0$\\\cline{2-7}
		& $ 2 $ & $ 16 $ & $ z_{2} $ & $ 0 $ & $ -z_{2} $ & $ 2z_{2}^{3}$\\\cline{2-7}
		& $3 $ & $ 16 $ & $ z_{3} $ & $ -z_{3}^{2} $ & $ 0 $ & $ z_{3}^{3}$\\\cline{2-7}
		& $ 4 $ & $ 4 $ & $ z_{4} $ & $ 0 $ & $ -z_{4} $ & $ 2z_{4}^3$\\\hline

\end{longtable}
\endgroup}


\begin{theorem}
\label{ThmMultRed} Let $E_T$ be as given in Table \ref{ta:ETmodel} for $T\neq C_3^0$. Then $E_{T}$ has multiplicative reduction at a prime $p$ with N\'{e}ron type $\mathrm{I}_n$ if and only if the parameters of $E_{T}$ satisfy one of the conditions listed in Table~\ref{ta:multred} at $p$. Moreover, $E_T$ has split or non-split multiplicative reduction at $p$ if the Type is labeled S or NS, respectively.
\end{theorem}

{\begingroup

\renewcommand{\arraystretch}{1.3}
 \begin{longtable}{C{0.5in}C{0.35in}C{1.9in}C{1.75in}C{0.35in}C{0.2in}}
	\hline
	$T$ & $p$ & $n$ where $n>0$ & Additional conditions & Type&$F_{T,j}$\\
	\hline
	\endfirsthead
	\hline
	$T$ & $p$ & $n$ where $n>0$ & Additional conditions & Type&$F_{T,j}$\\
	\hline
	\endhead
	\hline

	\multicolumn{4}{r}{\emph{continued on next page}}
	\endfoot
	\hline
	\caption{Type of multiplicative reduction of $E_T$}\label{ta:multred}
	\endlastfoot

$C_{3}$ & $\geq2$ & $3v_{p}\!\left(  b\right) $ & & S&$F_{T,1}$\\\cline{2-6}
& $\neq 3$ & $v_{p}\!\left(  a-27b\right)$ & $p\equiv1\ \operatorname{mod}6$ *
& S&$F_{T,2}$\\\cline{3-5}
&  & $v_{p}\!\left(  a-27b\right)$  & $p\equiv5\ \operatorname{mod}6$ or $p=2$ *&
NS&\\\hline

$C_{5}$ & $\geq2$ & $5v_{p}\!\left(  ab\right)  $ & & S&$F_{T,1}$\\\cline{2-6}
& $\geq7$ & $v_{p}\!\left(  a^{2}+11ab-b^{2}\right)$  & $\left(  \frac
{-5\left(  a^{2}+b^{2}\right)  }{p}\right)  =1$ * & S\\\cline{3-6}
&  & $v_{p}\!\left(  a^{2}+11ab-b^{2}\right)$  & $\left(  \frac{-5\left(
a^{2}+b^{2}\right)  }{p}\right)  =-1$ *& NS\\\hline

$C_{6}$ & $\geq2$ & $6v_{p}\!\left(  b\right)  $ & & S&$F_{T,1}$\\\cline{2-5}
& $\neq3$ & $2v_{p}\!\left(  a\right) $ & $p\equiv1\ \operatorname{mod}6$ * & S\\\cline{4-5}
&  && $p\equiv
5\ \operatorname{mod}6$ or $p=2$ *& NS\\\cline{2-6}

& $>3$ & $v_{p}\!\left(  a+9b\right)
$ & $p\equiv1\ \operatorname{mod}6$ * & S &$F_{T,2}$\\\cline{4-5}
&  && $p\equiv
5\ \operatorname{mod}6$ *& NS\\\cline{2-6}

& $\geq3$ & $3v_{p}\!\left(  a+b\right)  $ & $v_{3}\!\left(  a\right)
=0$ if $p=3$  &  S&$F_{T,3}$\\\cline{2-6}
& $2$  & $v_{2}\!\left(  a+9b\right)  -3$ & $v_{2}\!\left(  a+b\right)  =3$& NS&$F_{T,4}$\\\cline{3-6}
&  & $3v_{2}\!\left(  a+b\right)  -9$ & $v_{2}\!\left(  a+b\right)  >3$ & S&$F_{T,5}$\\\hline

$C_{7}$ & $\geq2$ & $7v_{p}\!\left(  ab\left(  a-b\right)  \right)  $ & & 
S&$F_{T,1}$\\\cline{2-6}
& $\geq 13$ & $v_{p}\!\left(  a^{3}+5a^{2}b-8ab^{2}+b^{3}\right)  $ & $\left(
\frac{-7\left(  a^{2}-ab+b^{2}\right)  }{p}\right)  =1$ * & S\\\cline{4-6}
&  &  & $ \left(
\frac{-7\left(  a^{2}-ab+b^{2}\right)  }{p}\right)  =-1$ * & NS\\\hline

$C_{9}$ & $\geq2$ & $9v_{p}\!\left(  ab\left(  a-b\right)  \right)  $ & &
S&$F_{T,1}$\\\cline{2-6}
& $\geq7$ & $3v_{p}\!\left(  a^{2}-ab+b^{2}  \right) +  $ & $p\equiv1\ \operatorname{mod}6$ *
& S\\\cline{4-6}
&  & $v_{p}\!\left(   a^{3}%
+3a^{2}b-6ab^{2}+b^{3} \right)  $ & $p\equiv5\ \operatorname{mod}6$ * &
NS\\\hline

$C_{10}$ & $\geq2$ & $v_{p}( a^{5}b^{10}\left(  a-2b\right)
^{5}\left(  a-b\right)  ^{10})  $ & $v_{2}\!\left(  a\right)  =0$
if $p=2$ & S&$F_{T,1}$\\\cline{2-6}
& $2$ & $v_{2}( a^{5}\left(  a-2b\right)  ^{5}) +$  $v_{2}\!\left(  a^{2}+2ab-4b^{2}\right) -12$
& $v_{2}\!\left(  a\right)  >0$ &
S&$F_{T,2}$\\\cline{2-6}
& $\geq 7$ & $v_{p}\!\left(  a^{2}+2ab-4b^{2}\right)  $ & $\left(  \frac
{b^{2}-a^{2}}{p}\right)  =1$ *& S\\\cline{4-6}
&  &  & $\left(  \frac{b^{2}-a^{2}%
}{p}\right)  =-1$ * & NS\\\cline{3-6}
&  & $2v_{p}\!\left(  a^{2}-3ab+b^{2}\right)  $ & $\left(  \frac{-\left(
ab+b^{2}\right)  }{p}\right)  =1$ * & S\\\cline{4-6}
&  &   & $\left(  \frac{-\left(
ab+b^{2}\right)  }{p}\right)  =-1$ * & NS\\\hline

$C_{12}$ & $\geq2$ & $v_{p}(  b^{12}\left(  a-b\right) 
^{12}(a-2b)^{6})  $ & $v_{2}\!\left(  a\right)  =0$
if $p=2$ or $v_{3}\!\left(  a\right)  =0$
if $p=3$ & S&$F_{T,1}$\\\cline{2-6}
& $2$ & $6v_{2}\!\left(  a-2b\right)  -6$ & $v_2(a)=1$& S&$F_{T,2}$\\\cline{3-5}
&  & $2v_{2}\!\left(  a\right)  -2$ & $v_2(a)\ge 2$&  NS&\\\cline{2-6}
& $\geq5$ & $2v_{p}(a) + v_{p}\!\left(  a^{2}-6ab+6b^{2}\right)   $ & $p\equiv
1\ \operatorname{mod}6$ *  & S&\\\cline{4-5}
&&& $p\equiv
5\ \operatorname{mod}6$  * & NS&\\\cline{3-6}
&  & $3v_{p}\!\left(   a^{2}-2ab+2b^{2}\right)+$ $ 4v_{p}\!\left(
a^{2}-3ab+3b^{2}\right)   $ & * & S\\\hline

$C_{2}\times C_{6}$ & $\geq2$ & $6v_{p}\!\left(  2a\left(  b-5a\right)
\left(  b-a\right)  \right)  $ & $\left(  i\right)  $ $v_{3}\!\left(
b\right)  =0$ if $p=3$ and $\left(  ii\right)  \ v_{2}\!\left(  a+b\right)
=0$ if $p=2$ & S&$F_{T,1}$\\\cline{2-6}
& $\geq5$ & $2v_{p}\!\left(  \left(  b^2-9a^2\right) 
\left(  b-9a\right)  \right)  $ & $p\equiv1\ \operatorname{mod}6$ * & S\\\cline{4-6}
&  &  & $p\equiv5\ \operatorname{mod}6$ * & NS \\\cline{2-6}
& $2$ & $ 6v_{2}(b-5a)+ 2v_{2}(b+3a)-24    $ & $v_2(a-b)=2$ and $ab-b^2 \equiv 4 \operatorname{mod} 16$ *  & NS&$F_{T,2}$ \\\cline{4-5}
& & & $v_2(a-b)=2$ and $ab-b^2 \equiv 12 \operatorname{mod} 16$ * & S \\\cline{3-6}
& &$ 6v_{2}( b-a)-18 $  & $v_2(a-b)\geq4$ & S&$F_{T,2}$ \\\cline{3-6}
&  & $2v_{2}(b-9a)-6$ &  $v_2(a-b)=3$ & NS&$F_{T,3}$ \\\cline{3-6}
& & $ 2v_{2}(b-3a) -2   $ & $v_2(a-b)=1$ & NS&$F_{T,4}$
\end{longtable}
\endgroup}
\vspace{-1em}
\begin{proof}
From Lemma~\ref{LemmaMult}, it is verified that $E_{T}$ has multiplicative reduction at a prime $p$ with
N\'{e}ron type I$_{n}$ if and only if the parameters of $E_{T}$ satisfy one of
the conditions for a prime $p$ and the listed $n>0$ in Table \ref{ta:multred}.
It remains to determine whether the type of multiplicative reduction is split
or non-split.

First, suppose $E_{T}$ has multiplicative reduction at a prime $p$ with
N\'{e}ron type I$_{n}$ where the additional condition corresponding to
$n$ in Table \ref{ta:multred} does not have a ``$\ast$''. In particular, we have a
Weierstrass model $F_{T,j}$ associated to each of these $n$, and it is
verified that $F_{T,j}$ is a minimal model
over $\mathbb{Q}_{p}$ with Weierstrass coefficients $a_{3},a_{4},a_{6}$ divisible by $p$. 
For each of the cases being considered, it is easily verified that $t^{2}+a_{1}t-a_{2}$ splits over $\overline{\mathbb{F}}_p$ where $a_{1}$ and $a_{2}$ are the
Weierstrass coefficients of the associated $F_{T,j}$. Next, we observe that for each of the cases considered in \eqref{NS}, $E_T$ has non-split multiplicative reduction at $2$ by Tate's Algorithm since $t^{2}+a_{1}t-a_{2} \equiv t^{2}+t+1\ \operatorname{mod}2$. For the remaining cases, Tate's Algorithm implies that $E_T$ has split multiplicative reduction since $t^{2}+a_{1}t-a_{2} \equiv t\left(  t+a_{1}\right)  \ \operatorname{mod}p$.
\begin{equation}
	\renewcommand{\arraystretch}{1}
	\renewcommand{\arraycolsep}{.15cm}%
\begin{array}
[c]{ccc}%
\toprule
T&n>0 & \text{Additional conditions}\\\toprule
C_{6}&v_{2}\!\left(
a+9b\right)  -3&v_{2}\!\left(
a+b\right)=3\\\hline
C_{12}&6v_{2}\!\left(  a\right)  -2 & v_{2}(a)  \ge2\\\hline
C_{2}\times C_{6}
& 2v_{2}(b-9a)-6&v_{2}\!\left(  a-b\right)  =3\\\cmidrule{2-3}
&2v_{2}(b-3a)-2&v_{2}\!\left(  a-b\right)  =1\\\bottomrule
\end{array}
\label{NS}
\end{equation}

It remains to consider the case when $E_{T}$ has N\'{e}ron type I$_{n}$ at $p$
with $n>0$ and the additional condition has a ``$\ast$'' present. 

\textbf{Case 1.} Suppose $T=C_{3}$ and let $a=c^{3}d^{2}e$ for $c,d,e$
integers such that $d,e$ are relatively prime squarefree positive integers.
Suppose $p$ is a prime such that $n=v_{p}\!\left(  a-27b\right)  >0$ with~$p\neq3$ and consider the Weierstrass model $F_{T,2}$. Let $a_{j}$ denote the
Weierstrass coefficients of~$F_{T,2}$. It is then verified that $p$ divides
$a_{3},a_{4},a_{6}$ and%
\[
t^{2}+a_{1}t-a_{2}=t^{2}-3cdet+3c^{2}d^{2}e^{2}.
\]
If $p=2$, then $v_2(a-27b)>0$ implies that $ab$ is odd. Thus $cde$ is odd and $t^{2}+a_{1}t-a_{2}\ \operatorname{mod}2$
is irreducible. By Tate's Algorithm $E_{T}$ has N\'{e}ron type I$_{n  }$ and non-split multiplicative reduction at $2$. If
$p\geq5$, then $t^{2}+a_{1}t-a_{2}$ splits modulo $p$ if and only if%
\[
\left(  \frac{a_{1}^{2}+4a_{2}}{p}\right)  =\left(  \frac{-3c^{2}d^{2}e^{2}%
}{p}\right)  =\left(  \frac{-3}{p}\right)  =1.
\]
It follows by Tate's Algorithm that $E_{T}$ has N\'{e}ron type I$_{n}$ and $E_{T}$ has split multiplicative reduction at $p$ if and only if $\left(  \frac{-3}{p}\right)=1.$ The theorem now follows for this case since the condition $\left(  \frac{-3}{p}\right)=1$ is equivalent to $p\equiv 1 \operatorname{mod}6$.

\textbf{Case 2. }Suppose $T=C_{6}$ and let $p$ be a prime such that
$v_{p}\!\left(  a\right)  >0$ (resp. $v_{p}\!\left(
a+9b\right)  >0$) with $p\neq3$ (resp. $p\geq5$). Under these assumptions, we consider the
Weierstrass model $F_{T,1}$ (resp. $F_{T,2}$) and let $a_{j}$ denote its
Weierstrass coefficients. Then%
\[
t^{2}+a_{1}t-a_{2}\equiv\left\{
\begin{array}
[c]{cl}%
t^{2}-bt+b^{2}\ \operatorname{mod}p & \text{if }v_{p}\!\left(  a\right)
>0\ \text{with\ }p\not =3,\\
t^{2}+3\left(  a+3b\right)  t+3\left(  a+3b\right)  ^{2}\ \operatorname{mod}%
p & \text{if }v_{p}\!\left(  a+9b\right)  >0\ \text{with\ }p\geq5.
\end{array}
\right.
\]
If $p\neq2$, then the type of multiplicative reduction is determined by
$\left(  \frac{-3}{p}\right)  $ since $t^{2}+a_{1}t-a_{2}$ splits in
$\mathbb{F}_{p}$ if and only if $\left(  \frac{-3}{p}\right)  =1$. Moreover,
if $v_{2}\!\left(  a\right)  >0$, then $b$ is odd and thus $E_{T}$ has non-split multiplicative
reduction at $2$.

\textbf{Case 3. }Suppose $T=C_{2}\times C_{6}$ and let $v_{2}\!\left(
a-b\right)  =2$. Then $a-b=4k$ for some odd integer $k$. Let $a_{j}$ denote
the Weierstrass coefficients of $F_{T,2}$. Then $a_{3},a_{4},a_{6}$ are even
integers and $a_{1}$ is odd. Moreover,%
\[
t^{2}+a_{1}t-a_{2}\equiv t^{2}+t-\frac{2k+3bk^{3}+k^{2}}{2}%
\ \operatorname{mod}2\equiv\left\{
\begin{array}
[c]{cl}%
t^{2}+t+1\ \operatorname{mod}2 & \text{if }bk\equiv1\ \operatorname{mod}4,\\
t\left(  t+1\right)  \ \operatorname{mod}2 & \text{if }bk\equiv
3\ \operatorname{mod}4.
\end{array}
\right.
\]
It follows that $E_{T}$ has split (resp. non-split) multiplicative reduction
at $2$ if $ab-b^{2}\equiv 12\ \operatorname{mod}16$ (resp. $ab-b^{2}%
\equiv 4\ \operatorname{mod}16$).

\textbf{Case 4.} It remains to consider the cases for which $E_{T}$ has
multiplicative reduction at a prime $p\geq5$ in
Table \ref{ta:multred}.
Next, let $c_{4}$ and $c_{6}$
denote the invariants associated to $E_{T}$. By \cite[Proposition
4.4]{MR2025384}, $E_T$ has split multiplicative reduction at $p$ if and only if $\left(  \frac{-c_{4}c_{6}}{p}\right)=1$.
The result now follows since
{\begingroup

\renewcommand{\arraystretch}{1.3}
 \begin{longtable}{ccc}
	\hline
	$T$ & $n>0$ & $-c_{4}c_{6}\ \operatorname{mod}p$\\
	\hline

	\endfirsthead
	\hline
	$T$ & $n>0$ & $-c_{4}c_{6}\ \operatorname{mod}p$\\
	\hline
	\endhead
	\hline

	\endfoot
	\hline
	\endlastfoot
$C_{5}$ & $v_{p}\!\left(  a^{2}+11ab-b^{2}\right)$   & $-5a^{4}b^{4}\left(
a^{2}+b^{2}\right)$  \\\hline
$C_{7}$ & $v_{p}\!\left(  a^{3}+5a^{2}b-8ab^{2}+b^{3}\right)   $ & $ -7a^6b^6(a-b)^6\left(
a^{2}-ab+b^{2}\right)$  \\\hline
$C_{9}$ & $3v_{p}\!\left(  a^{2}-ab+b^{2}\right)   $ & $ 243ab^{29}$\\\cline{2-3}
& $v_{p}\!\left(  a^{3}+3a^{2}b-6ab^{2}+b^{3}\right)   $ & $ -243\left(
ab\left(  a-b\right)  \right)  ^{10}$\\\hline
$C_{10}$ & $v_{p}\!\left(  a^{2}+2ab-4b^{2}\right)   $ & $ -2^{10}5^{2}\left(
ab-b^{2}\right)  ^{14}\left(  b^{2}-a^{2}\right)$  \\\cline{2-3}
& $2v_{p}\!\left(  a^{2}-3ab+b^{2}\right)   $ & $ -25\left(  ab-b^{2}\right)
^{14}\left(  ab+b^{2}\right)$  \\\hline
$C_{12}$ & $2v_p(a)$ & $-2^{10}3^5b^{40}$ \\\cline{2-3}
& $v_{p}\!\left(  a^{2}-6ab+6b^{2}\right)   $ & $ -2^{10}3^{5}%
b^{20}\left(  a-b\right)  ^{20}$\\\cline{2-3}
& $3v_{p}\!\left(  a^{2}-2ab+2b^{2}\right)   $ & $ 2^{10}b^{40}$\\\cline{2-3}
& $4v_{p}\!\left(  a^{2}-3ab+3b^{2}\right)   $ & $ 3^{5}b^{21}\left(
a-b\right)  ^{15}\left(  a-2b\right)  ^{4}$\\\hline
$C_{2}\times C_{6}$ & $ 2v_{p}\!\left(  b^{2}-9a^{2}\right)   $ & $ -2^{10}%
3^{5}a^{10}\left(  b-5a\right)  ^{10}$\\\cline{2-3}
& $2v_{p}\!\left(  b-9a\right)$   & $-2^{40}3^{5}a^{20}$
\end{longtable}\addtocounter{table}{-1}
\endgroup}\vspace{-2em}
\end{proof}

\section{Representations attached to \texorpdfstring{$E_T$}{}}\label{sectionreps}
Given a rational elliptic curve $E/\Q$ there is a cuspidal automorphic representation $\pi\cong\otimes_{p\le \infty}\pi_{p}$ of $\operatorname*{GL}\!\left(  2,\mathbb{A}_{\mathbb{Q}}\right)$ with trivial central character attached to it such that $\pi_p$ is the local representation of $\GL(2,\Q_p)$ associated to $E/\Q_p$ and $\pi_{\infty}$ is a holomophic discrete series representation of weight $2$. The \emph{conductor} $a(\pi)$ of an automorphic representation $\pi\cong\otimes_{p\le \infty}\pi_{p}$ of
$\operatorname*{GL}\!\left(  2,\mathbb{A}_{\mathbb{Q}}\right)$ is defined as $a(\pi)=\prod_p p^{a(\pi_p)}$. Moreover, $a(\pi)=N_E$ where $N_E$ is the conductor of $E$. 

In what follows we use the terminology from Section~\ref{local rep} on local representations of $\GL(2,\Q_p)$. Moreover, for a character $\chi$ of $\Q_p^{\times}$, the \emph{order of $\chi$}, denoted as ${\rm ord}(\chi)$, is the smallest positive integer $n$ such that $\chi^n=1$.

Recall that the families $E_T$ given in Table~\ref{ta:ETmodel} parametrize all elliptic curves with an odd non-trivial torsion point. In this section, we determine the cuspidal automorphic representations associated to $E_T/\Q$ in terms of its parameters. 
\subsection{\texorpdfstring{$3$}{}-Torsion point}
\begin{theorem}
\label{ThmforC3}Let $T=C_{3}, C_{3}^{0}$. When $T=C_{3}$, write $a=c^{3}d^{2}e$ where $d$ and $e$ are positive relatively prime squarefree integers. Let $\pi\cong \otimes_{p\le \infty} \pi_p$ be the cuspidal automorphic representation of $\GL(2,\A_{\Q})$ attached to $E_T$. Then, $\pi_{\infty}$ is a holomorphic discrete series representation of weight $2$ and the representations $\pi_p$ at each finite prime $p$ are given~by Table~\ref{TableforC3}.
\end{theorem}
\vspace{-0.5em}
{\begingroup

\renewcommand{\arraystretch}{1.1}
 \begin{longtable}{ccccc}
	\hline
	 $T$&$p$ & $\pi_p$ & Conditions on $a,b$\\
	\hline
	\endfirsthead
	\hline
	$T$&$p$ & $\pi_p$ & Conditions on $a,b$\\
	\hline
	\endhead
	\hline

	\multicolumn{4}{r}{\emph{continued on next page}}
	\endfoot
	\hline
	\caption{Representations for $E_{C_3}$ and $E_{C_3^0}$}\label{TableforC3}
	\endlastfoot
	
$C_3$&$\geq 2$ & $\chi\times \chi^{-1},\ a(\chi)=0$ & $v_p(deb\left(  a-27b\right))=0$ \\ \cmidrule{2-4}
&$\geq 2$ & $ {\rm St}_{\GL(2,\Q_p)} $ & $v_p(b)>0$ \\ \cmidrule{2-4}
&$\neq 3$ & ${\rm St}_{\GL(2,\Q_p)} $ &   $v_p\left(  a-27b\right)>0,  $ \\ 
&&&  $   p\equiv1\ \operatorname{mod}6 $  \\ \cmidrule{3-4}
&& $ (\gamma(E/\Q_p),\cdot){\rm St}_{\GL(2,\Q_p)} $ &  $ v_p\left(  a-27b\right)>0  $ \\ 
&&$a((\gamma(E/\Q_p),\cdot))=0$ &  $   p\equiv5\ \operatorname{mod}6 $ or $ p=2 $ \\ \cmidrule{2-4}
& $ 3k+1 $   & $ \chi\times \chi^{-1},\ a(\chi)=1, $   &  $ v_{p}\!\left(  a\right)  \equiv1, 2\ \operatorname{mod}3 $  \\
&&  $ {\rm ord}(\chi|_{\Z_p^{\times}})=3 $    & \\\cmidrule{2-3}
& $ 3k+2  $  &  $ \omega_{F,\xi},\ F/\Q_p\ {\rm unramified}$,   & \\
&& $ a(\xi)=1,\ {\rm ord}(\xi|_{\OF_F^{\times}})=3 $ & \\\cmidrule{2-4}
&$3$&  $ \omega_{F,\xi}   $ &  $ v_{3}\!\left(  a\right)  \equiv0\ \operatorname{mod}3,v_{3}\!\left(
a-27b\right)  =3,  $ \\
&&  $ F=\Q_3(\sqrt{\Delta_T})\ {\rm ramified}$, & $  bd^{2}e^{3}\left(  b^{3}d^{2}e^{5}-c\right)  \not \equiv
-2\ \operatorname{mod}9  $ &  \\\cmidrule{4-4}
&&  $ a(\xi)=2 $, & $  v_{3}\!\left(  a-27b\right)  =5  $  \\\cmidrule{4-4}
&& $ {\rm ord}(\xi|_{\OF_F^{\times}})=6 $ &  $ v_{3}\!\left(  a\right)  =1 $ \\\cmidrule{3-4}
&& $ \omega_{F,\xi},\ F=\Q_3(i) {\rm\ unramified}$, &  $ v_{3}\!\left(  a\right)  \equiv0\ \operatorname{mod}%
3,v_{3}\!\left(  a-27b\right)  =3, $ \\
&&  $  a(\xi)=1,\ {\rm ord}(\xi|_{\OF_F^{\times}})=4  $  &  $ bd^{2}e^{3}\left(  b^{3}d^{2}e^{5}-c\right)  \equiv-2\ \operatorname{mod}%
9  $ &  \\\cmidrule{3-4}
&&  $ \omega_{F,\xi},\ F=\Q_3(i) {\rm\ unramified}$, &  $ v_{3}\!\left(  a\right)  =2, ab \equiv 18 \operatorname{mod} 27$ \\
&& $ a(\xi)=2,\ {\rm ord}(\xi|_{\OF_F^{\times}})=6$ &\\\cmidrule{3-4}
&& $\chi\times \chi^{-1}, \ a(\chi)=2,$& $v_{3}\!\left(  a\right)  =2, ab \equiv 9\operatorname{mod}27$ \\
&&${\rm ord}(\chi|_{\Z_3^{\times}})=6$&  \\\cmidrule{3-4}
&& $\chi\times \chi^{-1},\ a(\chi)=1,  $ & $v_{3}\!\left(  a-27b\right)  =6$\\
&&  ${\rm ord}(\chi|_{\Z_3^{\times}})=2$ & \\\cmidrule{3-4}
&&$(\gamma(E/\Q_3),\cdot){\rm St}_{\GL(2,\Q_3)}$ & $n=v_{3}\!\left(  a-27b\right)  -6\geq1$\\
&&$a((\gamma(E/\Q_3),\cdot))=1$&\\\cmidrule{3-4}
&&  $\omega_{F,\xi},\ F=\Q_3(\sqrt{\Delta_T})\ {\rm ramified}$  & $v_{3}\!\left(  a\right)  \equiv1,2\ \operatorname{mod}3,\text{ }%
v_{3}\!\left(  a\right)  >2$ \\
&&$a(\xi)=4,\ {\rm ord}(\xi|_{\OF_F^{\times}})=6$&\\\midrule
$C_3^0$&$\geq 2$ & $\chi\times \chi^{-1},\ a(\chi)=0$ & $v_p(3a)=0$ \\ \cmidrule{2-4}
& $3k+1$  &$\chi\times \chi^{-1},\ a(\chi)=1,$  & $v_{p}\!\left(  a\right) =1, 2$ \\
&& ${\rm ord}(\chi|_{\Z_p^{\times}})=3$   & \\\cmidrule{2-3}
&$3k+2$  & $\omega_{F,\xi},\ F/\Q_p\ {\rm unramified}$  & \\
&&$a(\xi)=1,\ {\rm ord}(\xi|_{\OF_F^{\times}})=3$& \\\midrule
&3 &  $\omega_{F,\xi},\ F=\Q_3(\sqrt{\Delta_T})\ {\rm ramified}$  & $v_{3}\!\left(  a\right)  =0 \text{ and }$ \\
& &  $a(\xi)=2,\ {\rm ord}(\xi|_{\OF_F^{\times}})=6$  &$a\equiv\pm
1,\pm4\ \operatorname{mod}9$ \\\cmidrule{2-4}
& &  $\omega_{F,\xi},\ F=\Q_3(\sqrt{\Delta_T})\ {\rm unramified}$  & $v_{3}\!\left(  a\right)  =0\text{ and }$ \\
& &  $a(\xi)=1,\ {\rm ord}(\xi|_{\OF_F^{\times}})=4$  &$a\equiv\pm
2\ \operatorname{mod}9$ \\\cmidrule{2-4}
&&  $\omega_{F,\xi},\ F=\Q_3(\sqrt{\Delta_T})\ {\rm ramified}$  & $v_{3}\!\left(  a\right)  =1, 2$ \\
& &  $a(\xi)=4,\ {\rm ord}(\xi|_{\OF_F^{\times}})=6$  & 
\end{longtable}
\endgroup}
\vspace{-1em}
\begin{proof}
For $T=C_3$ or $C_3^0$, the minimal discriminant of $E_{T}/\mathbb{Q}_p$ are $\Delta_T=d^4e^8b^3 (a- 27b)$ and $-27a^4$ for $T=C_3$  and $C_3^0$, respectively. When $v_p(\Delta_T)= 0$, the representation $\pi_p$ is the unramified principal series, i.e., $\pi_p=\chi\times \chi^{-1}$ where $\chi$ is a character of  $\Q_p^{\times}$ with $a(\chi)=0$. Next, we observe that $E_{C_{3}^0}$ has 
$j$-invariant $0$ and thus $E_{C_{3}^0}$ has additive reduction at each prime dividing the minimal discriminant. So suppose $T=C_3$ and that $E_T$ has multiplicative reduction at $p$. 
By Theorem~\ref{ThmMultRed}, we have necessary and sufficient conditions on the parameters of $E_T$ to determine when $E_T$ has split (resp. non-split) multiplicative reduction at $p$. In this case, by \cite[Sec.\ 15]{Rohrlich1994} (also see \cite[Theorem~2.1.1]{roy2019paramodular}), the representation is $\pi_p={\rm St}_{\GL(2,\Q_p)}$  (resp. $(\gamma(E/\Q_p),\cdot){\rm St}_{\GL(2,\Q_p)}$).
Here $(\gamma(E/\Q_p),\cdot)$ is the unramified quadratic character of $\Q_p^{\times}$, where $(\cdot, \cdot)$ is the
Hilbert symbol. It remains to consider the cases when $E_T$ has additive reduction at $p$.

Suppose $T=C_3$ or $C_3^0$. Suppose further that $p\neq 3$ and $v_{p}\!\left(  a\right)  \equiv1, 2\ \operatorname{mod}3$. For these cases, 
the N\'{e}ron Type of $E_T$ is $\mathrm{IV}$ or $\mathrm{IV}^{*}$ by \cite[Theorems~3.4~and~3.5]{BarRoy}. Moreover,
$v_p(\Delta_T)=4$ or $8$.

Suppose $p\ge 5$.  Then $(p-1)v_p(\Delta_T) \equiv 0\mod 12 \text{ if and only if } p \equiv 1 \mod 3$. By \cite[Theorem 2.2.1]{roy2019paramodular}, when $p \equiv 1 \mod 3$, we have that
$\pi_p=\chi\times \chi^{-1}$ where $\chi$ is the character of $\Q_p^{\times}$ such that the order of $\chi|_{\Z_p^{\times}}$ is $3$ and $a(\chi)=1$. When $p \equiv 2 \mod 3$, \cite[Theorem~2.2.2]{roy2019paramodular} implies that $\pi_p=\omega_{F, \xi}$ where $F$ is the unramified extension over $\Q_p$ and $\xi$ is the character of $F/\Q_p$ with $a(\xi)=1$ and the order of $\xi|_{\OF_F^{\times}}$ is $3$.

Now consider $p=2$. We know $f_2=2$ by \cite[Theorem~3.5]{BarRoy}. By Lemma~\ref{rep of conductor 2}, there is a unique representation of $\GL(2,\Q_2)$ with conductor $2$. So, we have $\pi_2=\omega_{F, \xi}$ where $F=\Q_2(\sqrt{5})$ is the unramified extension over $\Q_2$ and $\xi$ is the character of $F/\Q_2$ with $a(\xi)=1$.  It follows that the order of $\xi|_{\OF_F^{\times}}$ is $3$ by \cite[Th\'eor\`eme 2]{Kraus1990}. 

It remains to consider the case when $E_T$ has additive reduction at $p=3.$ We proceed by cases below.

\textbf{Case 1.}
Suppose $(i)$ $T=C_3$ with  $v_{3}\!\left(  a\right)  \equiv0\ \operatorname{mod}3$,  $bd^{2}e^{3}\left(  b^{3}d^{2}e^{5}-c\right)  \not \equiv -2\ \operatorname{mod}9$, $v_{3}\!\left(a-27b\right)  =3$ or $(ii)$ $T=C_3^0$ with $v_{3}\!\left(  a\right)  =0\text{ and }a\equiv\pm1,\pm4\ \operatorname{mod}9$. By \cite[Theorems~3.4~and~3.5]{BarRoy}, the N\'{e}ron Type is $\mathrm{II}$ with $f_3=3$. By \cite[Theorem~2.5.3]{roy2019paramodular}, we get $\pi_3=\omega_{F,\xi}$ where $F=\Q_3(\sqrt{\Delta_T})$ is a ramified extension over $\Q_3$ and $\xi$ is the character of $F/\Q_3$ with $a(\xi)=2$ and the order of $\xi|_{\OF_F^{\times}}=6$. 

\textbf{Case 2.} 
Suppose $(i)$ $T=C_3$ with  $v_{3}\!\left(  a\right)  \equiv0\ \operatorname{mod}3$, $bd^{2}e^{3}\left(  b^{3}d^{2}e^{5}-c\right) \equiv -2\ \operatorname{mod}9$, $v_{3}\!\left(a-27b\right)  =3$ or $(ii)$ $T=C_3^0$ with $v_{3}\!\left(  a\right)  =0\text{ and }a\equiv\pm2\ \operatorname{mod}9$. Then, by \cite[Theorems~3.4~and~3.5]{BarRoy} the N\'{e}ron Type is $\mathrm{III}$ with $f_3=2$. Then, by \cite[Theorem~2.5.3]{roy2019paramodular}, we get $\pi_3=\omega_{F,\xi}$ where $F$ is the unramified extension over $\Q_3$ and $\xi$ is the character of $F/\Q_3$ with $a(\xi)=1$ and the order of $\xi|_{\OF_F^{\times}}=4$.

\textbf{Case 3.} If $(i)$ $T=C_3$ with $v_{3}\!\left(  a\right)  \equiv2\ \operatorname{mod}3$ with $v_{3}\!\left(  a\right) >2$ or $(ii)$ $T=C_3^0$ with $v_{3}\!\left(  a\right)  =1$, then by \cite[Theorems~3.4~and~3.5]{BarRoy}, the N\'{e}ron Type at $3$ is $\mathrm{IV}$ and $f_3=5$. Similarly, the N\'{e}ron Type at $3$ is $\mathrm{IV}^{\ast}$ with $f_3=5$ if $(i)$ $T=C_3$ with $v_{3}\!\left(  a\right)  \equiv1\ \operatorname{mod}3$ with $v_{3}\!\left(  a\right) >2$ or $(ii)$ $T=C_3^0$ with $v_{3}\!\left(  a\right)  =2$. By \cite[Theorem~2.5.3]{roy2019paramodular}, we have that $\pi_3=\omega_{F,\xi}$ where $F=\Q_3(\sqrt{\Delta_T})$ is a ramified extension over $\Q_3$ and $\xi$ is the character of $F/\Q_3$ with $a(\xi)=4$ and the order of $\xi|_{\OF_F^{\times}}=6$.

\textbf{Case 4.} Suppose $T=C_3$ and $v_{3}\!\left(  a-27b\right)=5$ (resp. $v_3(a)=1$). By \cite[Theorem~3.5]{BarRoy}, $f_3=3$ and the N\'{e}ron Type is $\mathrm{IV}$ (resp. $\mathrm{IV}^{\ast}$).
Then, by \cite[Theorem~2.5.3]{roy2019paramodular}, we get $\pi_3=\omega_{F,\xi}$ where $F=\Q_3(\sqrt{\Delta_T})$ is a ramified extension over $\Q_3$ and $\xi$ is the character of $F/\Q_3$ with $a(\xi)=2$ and the order of $\xi|_{\OF_F^{\times}}=6$.

\textbf{Case 5.} Suppose $T=C_3$ and $v_3(a)=2$. By \cite[Theorem~3.5]{BarRoy}, the N\'{e}ron Type is $\mathrm{IV}$ with $f_3=4$. Now observe that $\Delta_T \in \Q_3^{\times2}$ if and only if $ab \equiv 9 \mod 27$. 
When $ab \equiv 9 \mod 27$, by \cite[Theorem~2.4.1]{roy2019paramodular} we have $\pi_3=\chi\times \chi^{-1}$ with $a(\chi)=2$ and the order of $\chi|_{\Z_3^{\times}}=6$. When $ab \equiv 18 \mod 27$, by  \cite[Theorem~2.5.3]{roy2019paramodular} we get $\pi_3=\omega_{F,\xi}$ where $F=\Q_3(\sqrt{\Delta_T})$ is the unramified extension over $\Q_3$ and $\xi$ is the character of $F/\Q_3$ with $a(\xi)=2$ and the order of $\xi|_{\OF_F^{\times}}=6$.

\textbf{Case 6.} Suppose $T=C_3$ and $v_{3}\!\left(  a-27b\right)=6$. Then, by \cite[Theorem 3.5]{BarRoy} we get, the N\'{e}ron Type is $\mathrm{I}_0^{\ast}$ with $f_3=2$. By \cite[Theorem~2.5.3]{roy2019paramodular}, we have that $\pi_3=\chi\times \chi^{-1}$ with $a(\chi)=1$ and the order of $\chi|_{\Z_3^{\times}}=2$.

\textbf{Case 7.} Suppose $T=C_3$ and $v_{3}\!\left(  a-27b\right)-6\ge 1$. Then, by \cite[Theorem~3.5]{BarRoy} we get, the N\'{e}ron Type is $\mathrm{I}_n^{\ast}$ with $n=v_{3}\!\left(  a-27b\right)-6$ and $f_3=2$. By \cite[Theorem~2.1.1]{roy2019paramodular}, we have that $\pi_3=(\gamma(E/\Q_3),\cdot){\rm St}_{\GL(2,\Q_3)}$ with $a((\gamma(E/\Q_3),\cdot))=1$ by \eqref{conductor}.

\end{proof}
\subsection{6-Torsion point}
\begin{theorem}
Let  $T=C_{6}$ and $\pi\cong \otimes_{p\le \infty} \pi_p$ be the cuspidal automorphic representation of $\GL(2,\A_{\Q})$ attached to $E_T$. Then, $\pi_{\infty}$ is a holomorphic discrete series representation of weight $2$ and the representations $\pi_p$ at each finite prime $p$ are given~by Table~\ref{repC6}.
\end{theorem}
{\begingroup

\renewcommand{\arraystretch}{1.2}
 \begin{longtable}{ccc}
	\hline
	$p$  & $\pi_p$ & Conditions on $ a,b$ \\
	\hline

	\endfirsthead
	\hline
	$p$  & $\pi_p$ & Conditions on $a,b$\\
	\hline
	\endhead
	\hline

	\multicolumn{3}{r}{\emph{continued on next page}}
	\endfoot
	\hline
	\caption{Representations for $E_{C_6}$}\label{repC6}
	\endlastfoot
	
$p$ & $\chi\times \chi^{-1},\ a(\chi)=0$ & $v_p(ab(a + 9b)(a + b))=0$ \\ \midrule
		  $p$&$(\gamma(E/\Q_p),\cdot) {\rm St}_{\GL(2,\Q_p)},\ a\left((\gamma(E/\Q_p),\cdot)\right)=0$ & $\text{Type  NS in Table~\ref{ta:multred}}$\\ \midrule
		 &${\rm St}_{\GL(2,\Q_p)}$& $\text{Type  S in Table~\ref{ta:multred}}$\\ \midrule
		  $2$  & $\omega_{F,\xi},\ F/\Q_2\ {\rm unramified}$, $a(\xi)=1,\ {\rm ord}(\xi|_{\OF_F^{\times}})=3$    &  $v_{2}\!\left(  a+b\right)  =1,2 $  \\\midrule
		  $3$  &  $\omega_{F,\xi},\ F/\Q_3\ {\rm unramified}$,\ $ a(\xi)=1,\ {\rm ord}(\xi|_{\OF_F^{\times}})=4$  &  $v_{3}\!\left(  a\right)  =1$\\\midrule
		  &   $\chi\times \chi^{-1},\ a(\chi)=1,\ {\rm ord}(\chi|_{\Z_3^{\times}})=2$  &  $v_{3}\!\left(  a+9b\right)  =2,\ v_{3}%
		\!\left(  a\right)  =2$  \\\midrule
		  &  $(\gamma(E/\Q_3),\cdot){\rm St}_{\GL(2,\Q_3)}$   &  $v_{3}\!\left(  a+9b\right)  \geq
		3$,$\ v_{3}\!\left(  a\right)  =2$  \\\cmidrule{3-3}
		 & $a\left((\gamma(E/\Q_p),\cdot)\right)=1$ & $v_{3}\!\left(  a\right)  \geq3$	
\end{longtable}
\endgroup}
\vspace{-1em}
\begin{proof}
For $T=C_6$, the minimal discriminant of $E_{T}/\mathbb{Q}_p$ is $\Delta_T=a^{2}b^{6}(a+9b)(a+b)^{3}$. When $v_p(\Delta_T)= 0$, the representation $\pi_p$ is the unramified principal series i.e., $\pi_p=\chi\times \chi^{-1}$ where $\chi$ is character of  $\Q_p^{\times}$ with $a(\chi)=0$. 
By Theorem~\ref{ThmMultRed}, we have necessary and sufficient conditions on the parameters of $E_T$ to determine when $E_T$ has split (resp. non-split) multiplicative reduction at $p$. By \cite[Sec.\ 15]{Rohrlich1994} (also see \cite[Theorem~2.1.1]{roy2019paramodular}), the representation is $\pi_p={\rm St}_{\GL(2,\Q_p)}$  (resp. $(\gamma(E/\Q_p),\cdot){\rm St}_{\GL(2,\Q_p)}$). Here $(\gamma(E/\Q_p),\cdot)$ is the unramified quadratic character of $\Q_p^{\times}$, where $(\cdot, \cdot)$ is the
Hilbert symbol.

\textbf{Case 1.} Suppose $p=2$ and $v_{2}\!\left(  a+b\right) =1$ or $2$. Then by \cite[Theorem~3.8]{BarRoy}, the N\'{e}ron Type at $2$ is $\mathrm{IV}$ or $\mathrm{IV}^{*}$ with $f_2=2$.
By Lemma~\ref{rep of conductor 2}, there is a unique representation of $\GL(2,\Q_2)$ with conductor $2$. So, we have $\pi_2=\omega_{F, \xi}$ where $F=\Q_2(\sqrt{5})$ is the unramified extension over $\Q_2$ and $\xi$ is the character of $F/\Q_2$ with $a(\xi)=1$.  It follows from \cite[Th\'eor\`eme 2]{Kraus1990} that the order of $\xi|_{\OF_F^{\times}}$ is $3$.

\textbf{Case 2.} Suppose $v_{3}\!\left(  a\right)=1$. Then by \cite[Theorem~3.8]{BarRoy}, the N\'{e}ron Type is $\mathrm{III}$ with~$f_3=2$. Then, by \cite[Theorem~2.5.3]{roy2019paramodular}, we get $\pi_3=\omega_{F,\xi}$ where $F$ is the unramified extension over $\Q_3$ and $\xi$ is the character of $F/\Q_3$ with $a(\xi)=1$ and the order of $\xi|_{\OF_F^{\times}}=4$.
	
	\textbf{Case 3.} Suppose $v_{3}\!\left(  a\right)=2$ and $v_{3}\!\left(  a+9b\right)  =2$. Then, by \cite[Theorem~3.8]{BarRoy} we get, the N\'{e}ron Type is $\mathrm{I}_0^{\ast}$ with $f_3=2$. By \cite[Corollary 2.4.3]{roy2019paramodular}, we get $\pi_3=\chi\times \chi^{-1}$ with~$a(\chi)=1$ and the order of $\chi|_{\Z_3^{\times}}=2$.
	
	\textbf{Case 4.} Suppose $v_{3}\!\left(  a\right)=2$ and $v_{3}\!\left(  a+9b\right)  \ge 3$ or  $v_{3}\!\left(  a\right)  \geq3$. Then, by \cite[Theorem~3.5]{BarRoy} we get that $f_3=2$ and the N\'{e}ron Type is $\mathrm{I}_n^{\ast}$ with $n=v_{3}\!\left(  a+9b\right)-2\geq1$ or $v_{3}\!\left(  a\right)-4\geq2$, respectively. So, using \cite[Theorem~2.1.1]{roy2019paramodular}, we get $\pi_3=(\gamma(E/\Q_3),\cdot){\rm St}_{\GL(2,\Q_3)}$~with $a(\gamma(E/\Q_3),\cdot)=1$  by \eqref{conductor}.	
\end{proof}

\subsection{Other torsion subgroups}
\begin{theorem}
\label{ThmotherTs}
	Suppose $T$ has a non-trivial point of odd order and $T\neq C_3, C_6$. Let $\pi\cong \otimes_{p\le \infty} \pi_p$ be the cuspidal automorphic representation of $\GL(2,\A_{\Q})$ attached to $E_T$. Then $\pi_{\infty}$ is a holomorphic discrete series representation of weight $2$. If $ v_{p}(\Delta_T)=0$, then $\pi_p=\chi\times \chi^{-1}$ with  $a(\chi)=0$. If $E_T$ satisfies a condition for Type S in Table~\ref{ta:multred}, then $\pi_p={\rm St}_{\GL(2,\Q_p)}$ and if $E_T$ satisfies a condition for Type NS in Table~\ref{ta:multred}, then $\pi_p=(\gamma(E/\Q_p),\cdot){\rm St}_{\GL(2,\Q_p)}$ with $a((\gamma(E/\Q_p),\cdot))=0$. If $E_T$ has additive reduction at $p$, then $\pi_p$ is given as in Table~\ref{repOtherTs}.
\end{theorem}
	{\begingroup
\renewcommand{\arraystretch}{1.3}
 \begin{longtable}{cccc}
	\hline
	$T$&$p$  & $\pi_p$ & $\text{Conditions on  }a,b $  \\
	\hline

	\endfirsthead
	\hline
	$T$&$p$  & $\pi_p$ & $\text{Conditions on  }a,b $  \\
	\hline
	\endhead
	\hline
	\multicolumn{4}{r}{\emph{continued on next page}}
	\endfoot
	\hline
	\caption{Representations for $E_T$ with $T\neq C_3,\, C_6$}\label{repOtherTs}
	\endlastfoot
		$C_5$&$5$ & $\omega_{F,\xi},\ F/\Q_5\ {\rm unramified},\ a(\xi)=1,\ {\rm ord}\big(\xi|_{\OF_{F}^{\times}}\big)=6$& $v_5(a+18b)=1$\\ \cmidrule{3-4}
		&&$\chi\times \chi^{-1},\ a(\chi)=1,\  {\rm ord}(\chi|_{\Z_5^{\times}})=4$ &$v_5(a+18b)\ge2$\\ \midrule
		$C_7$&$7$ &$\chi\times \chi^{-1},\ a(\chi)=1,\ {\rm ord}(\chi|_{\Z_7^{\times}})=6$  &$v_{7}(a+4b)\ge 1$\\ \midrule
			$C_9$&$3$&$\omega_{F,\xi},\ F/\Q_3\ {\rm ramified},\ a(\xi)=2,\ {\rm ord}\big(\xi|_{\OF_{F}^{\times}}\big)=6$& $v_3(a+b)\ge1$\\ \midrule
			$C_{10}$&$5$&$\chi\times \chi^{-1},\ a(\chi)=1,\ {\rm ord}(\chi|_{\Z_5^{\times}})=4$ & $v_5(a+b)\ge 1$\\ \midrule
			 $C_{12}$&$3$&$(\gamma(E/\Q_3),\cdot){\rm St}_{\GL(2,\Q_3)},\ a(\gamma(E/\Q_3),\cdot)=1$  & $v_3(a)\ge1$\\\midrule
			  	  $C_{2}\times C_6$&$3$&$(\gamma(E/\Q_3),\cdot){\rm St}_{\GL(2,\Q_3)},\ a(\gamma(E/\Q_3),\cdot)=1$ & $v_3(b)\ge1$\\
		\bottomrule
\end{longtable}
\endgroup}
\vspace{-1em}

\begin{proof}
If $v_p(\Delta_T)= 0$ i.e., $E_T$ has good reduction at $p$, then the representation $\pi_p$ is the unramified principal series i.e., $\pi_p=\chi\times \chi^{-1}$ where $\chi$ is character of  $\Q_p^{\times}$ with $a(\chi)=0$. Table~\ref{ta:multred} provides necessary and sufficient conditions on the parameters of $E_T$ to determine when $E_T$ has split (resp. non-split) multiplicative reduction at $p$. By \cite[Sec.\ 15]{Rohrlich1994} (also see \cite[Theorem~2.1.1]{roy2019paramodular}), the representation is $\pi_p={\rm St}_{\GL(2,\Q_p)}$  (resp. $(\gamma(E/\Q_p),\cdot){\rm St}_{\GL(2,\Q_p)}$). Here $(\gamma(E/\Q_p),\cdot)$ is the unramified quadratic character of $\Q_p^{\times}$, where $(\cdot, \cdot)$ is the
Hilbert symbol. 

By \cite[Theorem~3.8]{BarRoy}, there are necessary and sufficient conditions to determine the N\'eron type and conductor exponent of $E_T$ at primes for which $E_T$ has additive reduction. In what follows, we assume this result implicitly. Below we consider the minimal discriminant $\Delta_T$ of $E_T$ as given in Table~\ref{mindiscs}.

\textbf{Case 1.} Let $T=C_5$. Then $E_T$ has additive
reduction at $p$ if and only if $p=5$ and $v_5(a+18b)\ge 1$. Moreover, the N\'eron type of $E$ at $p=5$ is type $\mathrm{II}$ (resp.\ type $\mathrm{III}$) when $v_5(a+18b)= 1$ (resp.\  $v_5(a+18b)\ge 2$). Then from \cite[Tableau I]{Papadopoulos1993} we have $4v_{5}(\Delta_T)\equiv0 \pmod {12}$ with
$e=\frac{12}{(12,v_{5}(\Delta_T))}=4$ if $v_5(a+18b)\ge2$ and $4 v_{5}(\Delta_T)\not \equiv 0
\pmod {12}$ with $e=\frac{12}{(12,v_{5}(\Delta_T))}=6$ if $v_5(a+18b)= 1$. Then, by \cite[Theorem~2.2.1 and Theorem~2.2.2]{roy2019paramodular} we get the described representations.

\textbf{Case 2.} Let $T=C_7$. Then $E_T$ has additive
reduction at $p$ if and only if $p=7$ and $v_7(a+4b)\ge 1$. Moreover,  the N\'eron type of $E$ at $p=7$ is type $\mathrm{II}$ when $v_7(a+4b)\ge 1$. Then from \cite[Tableau I]{Papadopoulos1993} we have $6v_{7}(\Delta_T)\equiv0 \pmod {12}$ with $e=\frac{12}{(12,v_{7}(\Delta_T))}=6$. Then, by \cite[Theorem~2.2.1]{roy2019paramodular} we get $\pi_7=\chi\times \chi^{-1}$ with $ a(\xi)=1$ and the order of  $\chi|_{\Z_7^{\times}}$ is~$6$.

\textbf{Case 3.} Let $T=C_9$. Then $E_T$ has additive
reduction at $p$ if and only if $p=3$ and $v_3(a+b)\ge 1$. Moreover, the N\'eron type of $E$ at $p=3$ is type $\mathrm{IV}$ and $f_3=3$ when $v_3(a+b)\ge 1$. 
Then, by \cite[Theorem~2.5.3]{roy2019paramodular} we get $\pi_3=\omega_{F,\xi}$ where $F=\Q_3(\sqrt{\Delta_T})$ is a ramified extension over $\Q_3$ and $\xi$ is the character of $F/\Q_3$ with $a(\xi)=2$ and the order of $\xi|_{\OF_F^{\times}}=6$.

\textbf{Case 4.} Let $T=C_{10}$. Then $E_T$ has additive
reduction at $p$ if and only if $p=5$ and $v_5(a+b)\ge 1$. In this case, the N\'eron type of $E$ at $p=5$ is type $\mathrm{III}$. Then from \cite[Tableau I]{Papadopoulos1993} we have $4 v_{5}(\Delta_T)\equiv0 \pmod {12}$ with
$e=\frac{12}{(12,v_{5}(\Delta_T))}=4$. Then, by \cite[Theorem~2.2.1]{roy2019paramodular} we get the described representations.

\textbf{Case 5.} Let $T=C_{12}$. Then $E_T$ has additive
reduction at $p$ if and only if $p=3$ and $v_3(a)\ge 1$. In the additive
reduction case, the N\'eron type of $E$ at $p=3$ is type $\mathrm{I}_n^{*}$ with $f_3=2$, where $n=2v_3(a)-1$. Then, by \cite[Sec.\ 15]{Rohrlich1994} (also see \cite[Theorem~2.1.1]{roy2019paramodular}) we get $\pi_3=(\gamma(E/\Q_3),\cdot){\rm St}_{\GL(2,\Q_3)} $ with $ a(\gamma(E/\Q_3))=1$  by \eqref{conductor}.

\textbf{Case 6.} Let $T=C_{2}\times C_6$. Then $E_T$ has additive
reduction at $p$ if and only if $p=3$ and $v_3(b)\ge 1$. Moreover, the N\'eron type of $E$ at $p=3$ is type $\mathrm{I}_n^{*}$ with $f_3=2$ when $v_3(b)\ge 1$. Then, by \cite[Sec.\ 15]{Rohrlich1994} (also see \cite[Theorem~2.1.1]{roy2019paramodular}), we get $\pi_3=(\gamma(E/\Q_3),\cdot){\rm St}_{\GL(2,\Q_3)} $ with $ a(\gamma(E/\Q_3))=1$  by \eqref{conductor}.
\end{proof}

\noindent \textbf{Acknowledgments.} We would like to thank Antonela Trbovi\'c for her helpful comments. We would also like to thank the referee for their detailed comments and suggestions.

\bibliographystyle{amsalpha}
\bibliography{LocalRepTorsion}

\end{document}